\documentclass{amsart}
\usepackage{mathptmx}
\usepackage{xcolor}
\usepackage{hyperref}
\usepackage[utf8]{inputenc}
\usepackage[margin=2.5cm]{geometry}
\usepackage[all]{xy}
\usepackage{amsmath,amssymb}
\usepackage{enumerate,multicol,stmaryrd}
\usepackage{mathrsfs}
\usepackage{amsaddr}
\usepackage{graphicx}
\graphicspath{ {./images/} }
\usepackage{pdfpages}

\newtheorem{ithm}{Theorem}[section]

\newtheorem{icor}[ithm]{Corollary}

\newtheorem{theorem}{Theorem}[section]

\newtheorem{proposition}[theorem]{Proposition}
\newtheorem{definition}{Definition}

\newtheorem{lemma}[theorem]{Lemma}

\theoremstyle{definition}

\newtheorem{question}{Question}

\DeclareMathOperator{\ga}{\textsl{g}}

\DeclareMathOperator{\Ricci}{Ric}

\def\rm#1{\mathrm{#1}}
\def\cal#1{\mathcal{#1}}
\def\bb#1{\mathbb{#1}}

\usepackage[all]{xy}

\newcommand{\halmos}{\hfill $\;\;\;\Box$\\}

\newcommand{\h}{\frac{1}{2}}

\DeclareMathOperator{\scal}{scal}

\renewcommand{\d}[1]{\ensuremath{\operatorname{d}\!{#1}}}

\def\rm#1{\mathrm{#1}}
\def\cal#1{\mathcal{#1}}
\def\bb#1{\mathbb{#1}}

\usepackage[all]{xy}

\begin{document}

\title{A direct approach to prescribing scalar curvature on bundles}

\author{Leonardo F. Cavenaghi}
\address{Instituto de Matemática, Estatística e Computação Científica -- Unicamp, Rua Sérgio Buarque de Holanda, 651, 13083-859, Campinas, SP, Brazil}
\email{leonardofcavenaghi@gmail.com}
	\author{Llohann D. Speran\c ca}
	\address{Instituto de Ciência e Tecnologia -- Unifesp, Avenida Cesare Mansueto Giulio Lattes, 1201,  12247-014, São José dos Campos, SP, Brazil}
	\email{speranca@unifesp.br}
		\thanks{Part of this work was developed while the first named author was a Ph.D. student at IME-USP financed by FAPESP 2017/24680-1. The second named author was financially supported by CNPq 131875/2016-7, 404266/2016-9.}
	
	\begin{abstract}
This note intends to demonstrate how to discuss scalar curvature functions' admissibility on bundles by directly applying some of the Kazdan--Warner results. Proofs of the concept include determining which functions are realizable as scalar curvature functions on the total space of several bundles over exotic manifolds. In addition, we employ traditional variational methods to provide reasonable conditions to smooth functions on the total spaces of some fiber bundles to be realized as scalar curvature functions for some Riemannian submersions metrics. We apply these last results to bundles over Calabi--Yau manifolds.
	\end{abstract}
\maketitle

\section{Introduction}

Understanding admissible geometric properties of some topological spaces has been a prolific way to comprehend the topology of the underlined manifold, see \cite{brendle2010ricci,perelman1,perelman2,perelman2002entropy}). However, the converse question, \textit{given a class of smooth manifolds, which are the admissible geometries on this class?} remains unsolved for almost every manifold. Moreover, there are compelling examples of exotic spheres $\Sigma^n$, first introduced By J. Milnor in \cite{mi}, which are homeomorphic to the standard sphere $S^n$ but not diffeomorphic, which adds relevance to this question.

 Concerning exotic spheres, much work has been done establishing the geometric properties of these manifolds. In the realm of positive/non-negative sectional curvature, we cite \cite{gromoll1974exotic,wilhelm-lots,gz,shankarannals,abresch2007wiedersehen,dearricott20117,grove2011exotic}; while for metrics of positive Ricci curvature we refer to \cite{nash1979positive, poor1975some, searle2015lift,wraith1997,wraith2007new,crowley2017intermediate, crowley2017positive, SperancaCavenaghiPublished, cavenaghi2019positive}).
 However, it is unknown if an exotic sphere with a metric of positive sectional curvature exists. More drastically, Hitchin proved that there are exotic spheres that do not even admit metrics of positive scalar curvature (see \cite{hitchin1974harmonic}). It is natural to ask
 \begin{question}\label{q:general}
To what extent does the smooth structure determine/obstruct the geometry?
 \end{question}

A first step in the direction of answering Question \ref{q:general} is the problem of \emph{which functions can be realized as the scalar curvature for some Riemannian metric on a closed connected manifold}. It had a significant development in the seminal work of Kazdan and Warner (see, e.g., \cite{kazdaninventiones, kazadanannals,kazdan1975}). On the other hand, despite furnishing general results, sometimes it conveys to produce explicit examples of which manifolds admit a given class of smooth functions as the scalar curvature of metrics there defined.

In this paper, we approach natural generalizations of this problem, namely:
\begin{enumerate}[(i)]
    \item Given a fiber bundle $F\hookrightarrow M \to B$ with compact total space and structure group and a smooth function $f: M \to \mathbb{R},$ can $f$ be realized as the scalar curvature of some Riemannian metric on $M$?
    \item Given a smooth $G$-principal bundle $G\hookrightarrow P\rightarrow B$ and a smooth $G$-invariant function $f : P \to \mathbb{R}$, can $f$ be realized as the scalar curvature of some $G$-invariant metric on $P$?
\end{enumerate}

 Submersions usually serve as a manufactory tool (mainly when interested in positive curvatures) when concerning constructing examples of Riemannian manifolds. In particular, many examples of exotic manifolds can be built as bases for certain submersions, see for instance \cite{SperancaCavenaghiPublished}. This paper deals with principal bundles and general fiber bundles with compact structure groups, examples of submersions. Collecting the fibers for such submersions, we arrive that their total spaces are regarded with a foliation (when fibers are connected).

 We convention now that the fibers of the here-treated submersions are named \emph{leaves}. Moreover, when there is no chance of confusion, a submersion shall be considered a pair $(M,\cal F)$, being $ M $ the total space and $ \cal F $ the collection of fibers.

\begin{ithm}\label{thm:scalar}
	Let $F \hookrightarrow M \stackrel{\pi}{\to} B$ be a fiber bundle where $M, F, B$ and the structure group $G$ are compact and connected. Assume that:
	\begin{enumerate}[$(i)$]
	    \item A principal orbit for the action of $G$ on $F$ has finite fundamental group,
	    \item $F$ carries a $G$-invariant metric such that $\Ricci_{F^{reg}/G}\geq 1.$
	    \end{enumerate}
	    Then
\begin{enumerate}[$(a)$]
    \item There is $\lambda \in (0,1]$, depending only on the geometry of the fiber, such that any smooth function $f: M \to \mathbb{R}$ satisfying $\dfrac{\min_{p\in M}f}{\max_{p\in M}f} \leq \lambda$ is the scalar curvature for some Riemannian metric on $M$, except maybe if $f = \mathrm{constant} \geq 0;$
    \item If $F$ has constant scalar curvature, then any smooth function $f : M \to \mathbb{R}$ is the scalar curvature of some Riemannian metric on $M$, except maybe if $f = \mathrm{constant} \geq 0;$
\end{enumerate}
\end{ithm}
We observe that whenever $G$-acts effectively on a smooth manifold $F$; there is an open and dense subset $F^{reg}\subset F$ such that every two points have conjugate isotropy subgroups. Given a $G$-invariant metric $\ga_F$ on $F$, the induced metric in $F^{reg}/G$, which happens to be a manifold under reasonable conditions, is named \emph{the orbit distance metric}. The hypotheses of Theorem \ref{thm:scalar} are natural because many examples can be built, as in \cite{searle2015lift}.

Theorem \ref{thm:scalar} is related to the Question \ref{q:general} in the following manner: Eells and Kuiper in \cite{ek} computed the number of $7$ (respectively $15$)-exotic spheres realized as total spaces of sphere bundles. Therefore, by setting $G = O(n+1),~n = 7,~15, F = S^n$, Theorem \ref{thm:scalar} enlightens Question \ref{q:general}:
\begin{theorem}[Grou--Rigas]\label{thm:grourigas}
 16 (resp. $4.096$) from the $28$ (resp. $16.256$) diffeomorphisms classes of the 
$7$-dimensional (resp. $15$)-exotic spheres are such that any real smooth function
$f : \Sigma^7 \to \mathbb{R}$ (resp. $f : \Sigma^{15} \to \mathbb{R})$ is the scalar curvature for some Riemannian metric on  $\Sigma^7$ (resp. $\Sigma^{15}),$ except maybe if $f = \mathrm{constant} \geq 0$.
\end{theorem}

Theorem \ref{thm:scalar} is a natural generalization of the following Theorem in \cite{mariaAlice} (and Theorem A in \cite{rigascalar}), from which Theorem \ref{thm:grourigas} also follows:
\begin{theorem}[Theorem A in \cite{rigascalar}]\label{thm:rigasgrou}
	Let $F \hookrightarrow M \stackrel{\pi}{\to} B$ be a fiber bundle with $M, F, B$, and structure group $G$ compact. Assume that:
	\begin{enumerate}[$(i)$]
	    \item $F$ carries a $G$-invariant metric with positive sectional curvature;
	    \item The action of $G$ ($\dim G \geq 2)$ on $F$ has only one orbit type.
	\end{enumerate}
	Then there exists $\lambda \in (0,1)$ such that any smooth function $f : M\to \mathbb{R}$ satisfying $\frac{\min f}{\max f} \leq \lambda$ is the scalar curvature for some Riemannian metric on $M$, except maybe if $f = \mathrm{constant} \geq 0.$
\end{theorem}

Theorem \ref{thm:exemplos1} below generalizes Theorem \ref{thm:grourigas} to several classes of bundles over exotic spheres and some connected sums, further achieving our goal, which concerns building examples. To state it well, we recall that sometimes specifying the base manifold for certain bundles is convenient. For this sake, we adopt the following convention: $F\hookrightarrow P\times_G F\to B$ denotes the associated bundle with fiber $F$ to the $G$-principal bundle $G\hookrightarrow P\to B$.

The constructions in Theorem \ref{thm:exemplos1} are generalizations of realizations of exotic spheres first obtained in Gromoll--Meyer \cite{gm}. These examples here provided are firstly presented in \cite{SperancaCavenaghiPublished} as a \emph{common ground} for \cite{DRS,gromoll1974exotic,sperancca2016pulling}: The manifolds to be considered are a geometric realization of `exotic' manifolds through isometric quotients of principal bundles over `standard' manifolds, being the realization incarnated in the form of a cross-diagram:
	\begin{equation}\label{eq:CDintro}
	\begin{xy}\xymatrix{& G\ar@{..}[d]^{\bullet} & \\ G\ar@{..}[r]^{\star} &  P\ar[d]^{\pi}\ar[r]^{\pi'} &M'\\ &M&}\end{xy}
	\end{equation}
 Usually, $M$ stands to a classic (standard) manifold and $M'$ to an exotic realization. Former instances of diagram \eqref{eq:CDintro} are found in \cite{duran2001pointed,DPR,DRS,speranca2016pulling}, mainly focused on the differential topology of exotic spheres. Unique properties of diagram \eqref{eq:CDintro} allow one to compare the geometries of $M$ and $M'$ (Proposition \ref{prop:preciso}).

\begin{ithm}\label{thm:exemplos1}
	Let $\Sigma^7$ and $\Sigma^8$ be any homotopy spheres of dimensions 7 and 8, respectively; $\Sigma^{10}$ be any homotopy 10-sphere which bounds a spin manifold;  $\Sigma^{4m{+}1},\Sigma^{8m+5}$ be Kervaire spheres of dimensions $4m{+}1,8m{+}5$, respectively. Then, there are explicit $G$-manifolds $P$ such that any smooth function $f$ (except possibly if $f = \mathrm{constant} \geq 0$) is the scalar curvature for some Riemannian  metric on the total space of the following bundles 
\begin{eqnarray*}
 S^{8r{+}k}\hookrightarrow P\times_GS^{8r{+}k}\rightarrow (M^{8r+k}\times N^{5-k})\#\Sigma^{8r+5},\\ 
	  \  S^l \hookrightarrow P\times_{G}S^l \rightarrow \Sigma^l, \\
	\     S^l\hookrightarrow P\times_{G}S^l\rightarrow M^l\#\Sigma^l 
			\end{eqnarray*}
	for $l=7,8,10,4m{+}1,~k= 0,1$, where
	\begin{enumerate}[$(i)$]
	    \item $M^7$ is any 3-sphere bundle over $S^4$,
	    \item $M^8$ is either a 3-sphere bundle over $S^5$ or a 4-sphere bundle over $S^4$,
	    \item $M^{10}=M^8\times S^2$ with  $M^8$  as in item $(ii)$,
     	\item  $M^{10}$ is any 3-sphere bundle over $S^7$, 5-sphere bundle over $S^5$ or 6-sphere bundle over $S^4$,
     	\item $M^{4m+1}\#\Sigma^{4m+1}$ where \label{item:5intro} 
	\end{enumerate}
	\begin{enumerate}[$(a)$]
				\item $S^{2m}\hookrightarrow M^{4m+1}\to S^{2m+1}$ is the sphere bundle associated to any  multiple\footnote{That is, a bundle whose transition function $\alpha: S^{n-1}\to G$ is a multiple of $\tau_{2m}: S^{2m}\to O(2m+1)$, $\tau_m^\bb C: S^{2m}\to U(m)$ or $\tau_m^\bb H: S^{4m+2}\to Sp(m)$, for $G=O(2m), U(m+1)$ or $Sp(m)$, the transition functions of the orthonormal frame bundle and its reductions, respectively.} of $O(2m{+}1)\hookrightarrow O(2m{+}2)\to S^{2m+1}$, the frame bundle of $S^{2m+1}$
				\item $\bb C \rm P^{m}\hookrightarrow M^{4m+1}\to S^{2m+1}$ is the $\bb C\rm P^m$-bundle associated to any  multiple of the bundle of unitary frames $U(m)\hookrightarrow U(m+1)\to S^{2m+1}$
				\item $M^{4m+1}=\frac{U(m+2)}{SU(2)\times U(m)}$	
	        	\item $N^{5-k}$ is any manifold with positive Ricci curvature and
	        	\begin{enumerate}[(d.i)]
				\item $S^{4r+k-1}\hookrightarrow M^{8r+k}\to S^{4r+1}$ is the $k$-th  suspension of the unitary tangent $S^{4r-1}\hookrightarrow T_1S^{4r{+}1}\to S^{4r+1}$,
				\item for $k=1$, $\bb H\rm P^{m}\hookrightarrow M^{8m+1}\to S^{4m+1}$ is the $\bb H\rm P^m$-bundle associated to any  multiple of $Sp(m)\hookrightarrow Sp(m+1)\to S^{4m+1}$
				\item for $k=0$, $M=\frac{Sp(m+2)}{Sp(2)\times Sp(m)},$
				\item for $k=1$, $M=M^{8m+1}$ is as in item  $(\ref{item:5intro}).$
				\end{enumerate}
			\end{enumerate}
\end{ithm}

\begin{ithm}\label{thm:exemplos2}
For the same manifolds considered in Theorem \ref{thm:exemplos1}, there is $\lambda \in (0,1)$ such that any smooth real function $f$ (except maybe if $f = \mathrm{constant} \geq 0$) defined in
			\begin{enumerate}[$(i')$]
			\item $P\times_G\Sigma^l,$
			    \item $P\times_GM^l\#\Sigma^l,$
			    \item $P\times_G(M^{8r+k}\times N^{5-k})\#\Sigma^{8r+5}$
			\end{enumerate}
			and it satisfies 
   \begin{equation*}\frac{\min f}{\max f} \leq \lambda,\end{equation*} is the scalar curvature of some Riemannian metric on the total space of the following bundles
			\begin{eqnarray*}
   	  \ (M^{8r+k}\times N^{5-k})\#\Sigma^{8r+5}\hookrightarrow P\times_G(M^{8r+k}\times N^{5-k})\#\Sigma^{8r+5}\rightarrow (M^{8r+k}\times N^{5-k})\#\Sigma^{8r+5},\\
			\    \Sigma^l \hookrightarrow P\times_G\Sigma^l \to \Sigma^l,\\
	\    M^l\#\Sigma^l\hookrightarrow P\times_{G}M^l\#\Sigma^l\rightarrow M^l\#\Sigma^l,
			\end{eqnarray*}
\end{ithm}
It is important to stress that $\lambda$ in the statement of Theorem \ref{thm:exemplos2} only depends on the geometry of a fixed Riemannian metric on the manifolds used as fibers on the associated bundles in the hypothesis. More precisely, the former Theorems \ref{thm:exemplos1} and \ref{thm:exemplos2} follow from the next results:
	\begin{ithm}[=Theorem \ref{thm:ajuda1}]\label{ithm:ajuda1}
	Let $M$ be a standard sphere with a metric of constant sectional curvature. Then any smooth function $f: P\times_GM\to \mathbb{R}$ (except maybe if $f = \mathrm{constant} \geq 0$) is the scalar curvature of some Riemannian metric on the total space of $M \hookrightarrow P\times_GM \to M'$ if, and only if, it is the scalar curvature of some Riemannian metric on the total space of $M \hookrightarrow P\times_GM \to M$.
	\end{ithm}
 \begin{ithm}[=Theorem \ref{thm:ajuda2}]\label{ithm:ajuda2}
Any smooth function $f: P\times_GM \to \mathbb{R}$ can be identified with a smooth function on $ P\times_GM'$. Moreover, if $f$ satisfies
 \[\frac{\min_{P\times_GM'} f}{\max_{P\times_GM'}f} \leq \frac{\min_{M'}\mathrm{scal}_{\ga'}}{\max_{M'}\mathrm{scal}_{\ga'}},\] then it is the scalar curvature of some Riemannian metrics on $P\times_GM$ and $P\times_GM'$, except maybe if $f = \mathrm{constant} \geq 0$.
 \end{ithm}

In contrast to the former discussion, we recall that in a fiber bundle $\pi:\overline M\to M$, a partial conformal change, called \textit{general vertical warping}, is in place. Given a \emph{basic} function $\phi:\overline M\to \bb R$, the general vertical warping defined by $\phi$ is defined by multiplying by $e^{2\phi}$ the component of the metric tangent to the fibers (see the Appendix \ref{sec:appendix} for further details). Common examples of fiber bundles are vector bundles and warped products, which are often non-compact. We recall that the problem of prescribing scalar curvature of complete metrics on non-compact manifolds is still open. Considering this, we use traditional variational methods to prove Theorem \ref{ithm:warped} below, which extends the work of Ehrlich--Yoon-Tae--Kim \cite{ehrlich1996}. In particular, we realize a great range of basic scalar curvature functions as warped products, including the case of non-compact fibers. There, the authors focus on prescribing constant scalar curvature.

Given a fiber bundle $(\overline M,\cal F)$, a smooth function $f: \overline M \rightarrow \mathbb{R}$ is said to be \emph{basic} if it is constant along the leaves of $\cal F$. Whenever we regard $(\overline M,\cal F)$ with a Riemannian submersion metric $\overline {\textsl g}$ with totally geodesic fibers for which each leaf has constant scalar curvature, the scalar curvature function $\mathrm{scal}_{\overline {\textsl g}}$ is basic, this justifies our assumptions of the admissible functions in the following:

\begin{ithm}\label{ithm:warped}
	Let $(M,\textsl{h})$ be a closed $n$-dimensional Riemannian manifold and $F$ be a $k$-dimensional manifold with a (complete) metric $\ga_F$ with constant scalar curvature $c$. Given a basic smooth function $f : M\times F\rightarrow \mathbb{R}$, assume one of the following conditions:
	\begin{enumerate}[$(1)$]
		\item $c> 0$ and $\scal_{\textsl{h}}>0$;
		\item $c \leq 0$ and $f : M\times F\to \mathbb R$ satisfies the condition: 
		\begin{gather*}
			- \left(\dfrac{k+1}{8k}\right)\left(f-\mathrm{scal}_{\textsl{h}}\right) + \left(\dfrac{(k+1)^2}{8(k-1)k}\right)c\left(\mathrm{vol}(M)\right)^{\tfrac{2}{k-1}} \geq 0.
		\end{gather*} 
	\end{enumerate}
	Then, 
	\begin{enumerate}[(a)]
	 \item 	if $|c| > 0$ there are $\mu_1,\mu_2>0$ and a basic smooth function $\phi : M \rightarrow \mathbb{R}$ for which the warpd product metric $\bar{\ga}_{\phi}$ obtained from $(M\times F,\mu_1\textsl{h}{\times}\mu_2\ga_F)$ satisfies $\scal_{\bar{\ga}_{\varphi}}(x,y)=f(x)$. In particular, if $M$ admits a metric with constant scalar curvature $c'$, $M\times F$ admits a warped product metrics with scalar curvature $c'$.
	 \item 	If $c =0$ there exists $A\geq 0$ such that the previous conclusion holds to $f+A$.
	\end{enumerate}
\end{ithm}
Despite stating in this manner, we prove a more general statement to Theorem \ref{ithm:warped} indeed, ensuring that the result holds for Riemannian submersions which are local products, changing the word ``warped product metrics'' on the thesis in the former statements, to ``general vertical warping metrics''. See Proposition \ref{thm:maingeral}.

A family of manifolds fundamental to studying scalar and Ricci curvature consists of the \textit{Calabi--Yau} manifolds. We call a compact K\"ahler manifold {Calabi--Yau} if its first Chern class vanishes, or, equivalently if it admits a Ricci flat K\"ahler metric. For simplicity, we call such a metric Calabi--Yau. To the authors' knowledge, Theorem \ref{ithm:warped} is the first result of prescribing scalar curvature functions on non-compact fiber bundles over Calabi--Yau manifolds. On the other hand, Tosatti--Zhang \cite{tosatti} proved that holomorphic submersions from compact Calabi--Yau manifolds are trivial, up to cover. We combine Theorem \ref{ithm:warped} with this result of Tosatti--Zhang to obtain:

\begin{icor}\label{icor:CY}
	Suppose $\pi:X\to Y$ is a holomorphic submersion with connected fiber $F$ satisfying one of the following:
	\begin{enumerate}
		\item $X,Y$ are projective manifolds with $X$ Calabi--Yau;
		\item $X,Y$ are compact K\"ahler manifolds; $Y,F$ are Calabi--Yau; and either $b_1(F)=0$ or $b_1(Y)=0$ and $F$ is a torus.
	\end{enumerate}
	If $f: Y\to \bb R$ is a scalar curvature function in $Y$, then there is a metric $\tilde{\ga}$ on $X$ whose scalar curvature is $f\circ\pi$. Moreover, if $\ga$ is a Calabi--Yau metric on $Y$, $\tilde{\ga}$ can be chosen so that $\pi$ is a Riemannian submersion over $\ga$.
\end{icor}

\section*{Terminology and notations}

	We denote by $R_{\ga}$ the Riemannian tensor of the metric $\ga$: 
	\[R_{\ga}(X,Y)Z=\nabla_X\nabla_YZ-\nabla_Y\nabla_XZ-\nabla_{[X,Y]}Z,\]
	where $\nabla$ stands for the Levi-Civita connection of $\ga$. So the adopted convetion to the sctional curvature of $\ga$ is $K_{\ga}(X, Y)=\ga(R_{\ga}(X, Y)Y, X)$. We sometimes denote $R_{\ga}(X, Y) = R_{\ga}(X, Y, Y, X),$ making it clear in the context, the unreduced sectional curvature of $\ga$. We adopt the standard decomposition notation $TM = \cal V\oplus \cal H$ for any foliated manifold $M$, being $\cal V$ the subbundle collecting tangent spaces to the leaves of the foliation $\cal F$ and $\cal H$ to some subbundle complementary choice in $TM$.
	
	When appearing as superscripts, $\mathcal{V}$ and $\mathcal{H}$ denote the projection of the underlined quantities on such subbundles, $\cal V$ and $\cal H$, respectively. Whenever we say we have a \emph{Riemannian principal bundle}, we mean that the principal bundle is considered with a Riemannian submersion metric. We recall that $A_XY:= \h[X,Y]^{\ker \d \pi}$ stands for the \emph{O'Neill integrability tensor} and, fixed any $X\in \cal H$, we denote by $A^*_X$ its $\ga$-dual. Fixed $X\in \cal H$, we denote by $S_X$ the \emph{shape operator} of the submersion $\pi$. Namely, if $\sigma : \ker d\pi \times \ker \d\pi \rightarrow \cal H$ denotes the \emph{second fundamental form} associated with $\pi$, then $S_X$ is the $\ga$-dual
\[\ga(S_X(\cdot),\cdot) = \ga(\sigma(\cdot,\cdot),X).\]

\section{Prescribing scalar curvature on fiber bundles with compact structure group}

\label{sec:compactgroup}

In this section, we prove Theorems \ref{thm:scalar} and \ref{thm:exemplos1}. To the proof of Theorem \ref{thm:scalar}, we make a simple and straight application of the following result by Kazdan and Warner:
\begin{theorem}[Theorem A in \cite{kazdaninventiones}]\label{thm:kdz}
Let $(M,g)$ be a compact manifold. Denote by $\mathrm{scal}_g$ the scalar curvature of $g$. Let $f\in C^{\infty}(M)$ be a smooth function on $M$. If there exists a constant $c > 0$ such that
\begin{equation}
    \min cf < \mathrm{scal}_g(p) < \max cf,~\forall p\in M,
\end{equation}
then there exists a Riemannian metric $\widetilde g$ on $M$ such that
\[\mathrm{scal}_{\widetilde g} = f.\]
\end{theorem}

\begin{proof}[Proof of Theorem \ref{thm:scalar}]
The hypotheses of the theorem are the same as Theorem A in \cite{searle2015lift}, ensuring that $F$ carries a $G$-invariant Riemannian metric $g_F$ with positive Ricci curvature. Given any Riemannian metric $g_B$ on $B$, consider on $M$ the unique Riemannian submersion metric $g$ such that its fibers are totally geodesic submanifolds isometric to $(F,g_F)$ (see \cite[Proposition 2.7.1, p. 97]{gw}). We reinforce that, in the case of principal bundles, this metric can be made $G$-invariant by defining a Kaluza-Klein $G$-invariant metric on the total space (see \cite[Proposition 5.1, p. 27]{SperancaCavenaghiPublished}).

Fix $p\in M$ and let $\{e_i\}_{i=1}^k$ be an orthonormal base for $\mathcal H_p$ and $\{e_j\}_{j=k+1}^n$ be an orthonormal base for $\mathcal V_p.$ Note that $\{e_i\}_{i=1}^k\cup \{e^{-t}e_j\}_{j=k+1}^n$ is a $g_t$-orthonormal base to $T_pM.$ Using the formulae given by Proposition \ref{prop:curvatures} we obtain an expression for the scalar curvature of $g_t$:
\begin{equation*}
  \mathrm{scal}_t(p) = \mathrm{scal_t}^{\mathcal H}(p) + 2e^{2t}\sum_{i,j}|A^*_{e_i}e_j|^2_g + e^{-2t}\mathrm{scal}_F(p).
\end{equation*}
Moreover,
\begin{equation*}
  \mathrm{scal_t}^{\mathcal H}(p) = \mathrm{scal}_B(p)(1-e^{2t}) + e^{2t}\mathrm{scal}_g^{\mathcal H}(p).
\end{equation*}

Denote by $s_t := \min_{p\in M}\mathrm{scal}_t(p)$ and by $S_t := \max_{p\in M}\mathrm{scal}_t(p).$ Then note that
\begin{equation*}
  \lambda := \lim_{t\to-\infty}\frac{s_t}{S_t} = \frac{\min_{p\in M}\mathrm{scal}_F}{\max_{p\in M}\mathrm{scal}_F} \leq 1.
\end{equation*}
If $f$ is not constant neither change sign, for each $\lambda' < \lambda$ one can find $t > 0$ such that
\[\lambda' =: \frac{\min f}{\max f} < \frac{s_{-t}}{S_{-t}} \leq 1,\]
so the result follows from Theorem \ref{thm:kdz}. Indeed, since $S_{-t} > 0$ we have $s_{-t} > \lambda'S_{-t}.$ Therefore,
\[S_{-t} \geq \mathrm{scal}_t \geq s_{-t} > \frac{\min f}{\max f}S_{-t}.\]
Hence, if $\min f > 0$ one gets directly
\begin{equation*}
  \max f > \mathrm{scal}_t > \min f.
\end{equation*}
On the contrary, if $\max_M f < 0$, then $f$ is always negative so changing $\min f$ to $-\min f$ and $\max f$ to $-\max f$ one gets
\begin{equation}
 -\max f > \mathrm{scal}_t > -\min f
\end{equation}
and the result follows similarly since $\max f$ is a negative number, say $-d$, but $\min f$ is another negative number, for instance, $-e$. Taking $c = 1$ ensures the result.

The result is standard for the case where $f$ is a negative constant since every negative constant is admissible as the scalar curvature for a Riemannian metric on any closed manifold, see Theorem C in \cite{kazdaninventiones}. \qedhere
\end{proof}

We now prove Theorems \ref{thm:exemplos1} and \ref{thm:exemplos2}. Our approach here is quite general, being the examples described in more detail in Section \ref{sec:examples}. Our constructions rely on the ones in \cite{cavenaghi2019positive, SperancaCavenaghiPublished}.	

Consider a compact connected principal bundle $G\hookrightarrow P \to M$ with a principal action $\bullet.$ Assume that there is another action on $P$, which we denote by $\star$,  that commutes with $\bullet$. This makes $P$ a $G\times G$-manifold. If one assumes that $\star$ is free, one gets a \textit{$\star$-diagram} of bundles:
	\begin{equation}\label{eq:CD}
	\begin{xy}\xymatrix{& G\ar@{..}[d]^{\bullet} & \\ G\ar@{..}[r]^{\star} & P\ar[d]^{\pi}\ar[r]^{\pi'} &M'\\ &M&}\end{xy}
	\end{equation}
	In \eqref{eq:CD}, $M$ is the quotient of $P$ by the $\bullet$-action and $M'$ is the quotient of $P$ by the $\star$-action. 
	
	Once $\bullet$ and $\star$ commute, $\bullet$ descends to an action on $M'$ and $\star$ descends to an action on $M$. We denote the orbit spaces of these actions by $M'/\bullet$ and $M/\star$, respectively. Corollary 5.2 in \cite{SperancaCavenaghiPublished} implies that one can choose a Riemannian metric $g'$ on $M'$ such that the orbit spaces $M/\star$ and $M'/\bullet$ are isometric, as metric spaces. Furthermore, it can be shown that the orbits of the $\star$-action on $M$ have a finite fundamental group if, and only if, the orbits of the $\bullet$-action does \cite[Theorem 6.4]{SperancaCavenaghiPublished}. This implies that if the $G$-manifold $M$ satisfies the hypotheses of $F$ in Theorem \ref{thm:scalar}, then $M'$ also does. We state it as a proposition for further reference:
 \begin{proposition}\label{prop:preciso}
     Let $M\leftarrow P \rightarrow M'$ be a diagram of bundles such as \eqref{eq:CD}. Then $M$ satisfies the hypotheses of $F$ in Theorem \ref{thm:scalar} if, and only if, $M'$ does so. That is, there exists a $G$-invariant metric $\ga$ for which $G$ acts on $M$ as $G$ acts on $F$ as in the hypotheses of Theorem \ref{thm:scalar} if, and only if, there exists a $G$-invariant metric $\ga'$ on $M'$ for which the same holds.
 \end{proposition}

	The idea of considering diagrams like \eqref{eq:CD} is two-fold: $M$ and $M'$ can be taken as homeomorphic manifolds that are not diffeomorphic to each other. Moreover, their roles can be interchanged to constructed fiber bundles (associated fiber bundles in this case), taking either $M$ or $M'$ to play the role of $F$ in Theorem $A$. More precisely, one can consider the following associated bundles
	\begin{enumerate}
\item The associated bundle $M \hookrightarrow P\times_GM \to M$ to $\pi : P\to M,$
\item The associated bundle $M' \hookrightarrow P\times_GM' \to M$ to $\pi : P \to M,$
\item The associated bundle $M \hookrightarrow P\times_GM \to M'$ to $\pi' : P\to M',$
\item The associated bundle $M' \hookrightarrow P\times_GM' \to M'$ to $\pi' : P\to M'$.
\end{enumerate}

 Recall that if $G\hookrightarrow P\rightarrow B$ is a principal $G$-bundle, and $F$ is a smooth manifold with an effective smooth action (on the left) by $F$, one can define an action of $G$ on $P\times F$ as:
\begin{equation}\label{eq:starassociated}g\ast (p,f) := (pg,g^{-1}f),~(p,f)\in P\times F\end{equation}
The associated fiber bundle $E=P\times_G F$ is defined then as a quotient of $P\times F$ by the relation $\sim$:
$$E = P\times F / \sim,$$
that is, it corresponds to the orbit space for the $\ast$-action given in equation \eqref{eq:starassociated}. Therefore, if $f: P\times M \to \mathbb{R}$ is a smooth $G$-invariant function for the $\ast$-action, it admits restrictions to the total spaces of $M \hookrightarrow P\times_GM \to M$ and $M \hookrightarrow P\times_GM \to M'$. Moreover, once on Theorem \ref{thm:exemplos1}, the manifolds $M$ are always standard spheres with metrics of constant sectional curvature; for functions obtained in this way, one has the following:
	\begin{theorem}\label{thm:ajuda1}
	Let $M$ be a standard sphere with a metric of constant sectional curvature. Then any smooth function $f: P\times_GM\to \mathbb{R}$ (except maybe if $f = \mathrm{constant} \geq 0$) is the scalar curvature of some Riemannian metric on the total space of $M \hookrightarrow P\times_GM \to M'$ if, and only if, it is the scalar curvature of some Riemannian metric on the total space of $M \hookrightarrow P\times_GM \to M$.
	\end{theorem}
 \begin{proof}
Note that to build the associated bundle $P\times_GM\rightarrow M$  we are considering $\pi : P\rightarrow M$ and so, the $\bullet$-action on $P$ and the $\star$-action on $M$. Similarly, to build the associated bundle $P\times_GM\rightarrow M'$, we are considering it as an associated bundle to $\pi': P\rightarrow M'$, and so, we see both $P$ and $M$ with the $\star$-action (being such action on $M$ the restricted by the projection $\pi$ from $P$). In any case, we have a bijection between the ring of smooth functions on $P\times_GM$, independently of which base this total space fiber over, with the ring of smooth functions on $P$ that are $G\times G$-invariant: Once the $\bullet$ and $\star$-actions commute on $P$, we can see it as a $G\times G$-manifold and invariant functions on $M$ can be lifted to $G\times G$-invariant functions on $P$, that can also be projected to $G$-invariant functions on $M'$.

Since $M$ is always taken as fiber and fits in the item $(b)$ of the thesis in Theorem \ref{thm:scalar}, we conclude the result. \qedhere
 \end{proof}

 According to Proposition \ref{prop:preciso}, the roles of $M$ and $M'$ can be interchanged as fibers for associated bundles to each $\pi: P\rightarrow M$ and $\pi': P\rightarrow M'$. Using once more that the $\bullet$ and $\star$-actions commute on $P$, we can see $P$ as a $G\times G$-manifold, so requiring that a smooth function $f$ is invariant under the $\bullet$ action on $M'$ and by the $\star$ action on $M$ is equivalent to require that such a smooth function is invariant by the $G\times G$-action on $P$. More precisely, we have an equivariant ring homomorphism of smooth functions:
\begin{equation}
C^{\infty}_{G\times G}(P) \cong C^{\infty}_{G}(M)\cong C^{\infty}_{G}(M').
\end{equation}
However, just as remarked in the proof of the former result, we also have $C^{\infty}(P\times_GF) \cong C^{\infty}_{G\times G}(P)$ for $F = M, M'$. Hence, we can indiscriminately identify any smooth function $f: P\times_GM \to \mathbb{R}$ with a smooth function $f: P\times_GM'\to \mathbb{R}$. More interesting, one may wonder whether admissible scalar curvature functions to Riemannian metrics on $P\times_GF$ can be constrained by admissible functions on $M$ or $M'$. This is the content of the following theorem. 
	\begin{theorem}\label{thm:ajuda2}
	Any smooth function $f: P\times_GM \to \mathbb{R}$ can be identified with a smooth function on $ P\times_GM'$. Moreover, if $f$ satisfies
 \[\frac{\min_{P\times_GM'} f}{\max_{P\times_GM'}f} \leq \frac{\min_{M'}\mathrm{scal}_{\ga'}}{\max_{M'}\mathrm{scal}_{\ga'}},\] then it is the scalar curvature of some Riemannian metrics on $P\times_GM$ and $P\times_GM'$, except maybe if $f = \mathrm{constant} \geq 0$.
	\end{theorem}
 \begin{proof}
     The idea consists of observing that such a $f$ can be identified with a smooth function on $M'$. Moreover, the condition in the hypothesis is that $f$ satisfies the same \emph{pinching condition} as $\mathrm{scal}_{\ga'}$. In this manner, we can always find $c >0$ such that (compare with the proof of Theorem \ref{thm:scalar})
     \begin{equation}
         c\min_{P\times_GM'}f < \mathrm{scal}_{\ga'} < c\max_{P\times_GM'}f,
     \end{equation}
     so the result is ensured by Theorem \ref{thm:scalar}. \qedhere
 \end{proof}

 Stating Theorem \ref{thm:ajuda2} in such a manner, i.e., constraining the conditions to $f$ with conditions to the metric $\ga'$ in $M'$ works as a manner to try distinguishing admissible geometries of exotic manifolds just by looking to admissible scalar curvature functions on these. Indeed, observe that the left side of the hypothesis can be changed to
 \begin{equation*}
     \frac{\min_{P\times_GM} f}{\max_{P\times_GM}f},
 \end{equation*}
 so we are comparing explicitly though indirectly, the geometry on $M$ with the geometry of $M'$.
	
	Theorem \ref{thm:ajuda1} and \ref{thm:ajuda2} finish the proofs of Theorems \ref{thm:exemplos1}, \ref{thm:exemplos2} since the mentioned examples are constructed using cross diagrams such as \eqref{eq:CD} in \cite[Theorem A]{SperancaCavenaghiPublished}. Also see Section \ref{sec:examples} for further details.

\section{Prescribing scalar curvature via general vertical warpings}
\label{sec:bundles}

Our goal in this section is to study how to prescribe scalar curvature in fiber bundles which are locally isometric to products. To this aim, we resort to a natural generalization of the \emph{canonical variation}, called \textit{general vertical warping}. Given a Riemannian submersion $\pi:(\overline{M},\ga)\to (M,\textsl h)$ and a smooth function $\phi:M\to \bb R$, we define the metric $\ga_\phi$ along the same lines as the canonical variation:

\[\ga_{\phi}= \ga\Big|_{\cal H{\times}\cal H} + e^{2\phi}\ga\Big|_{\cal V{\times}\cal V}. \]
The submersion $\pi:(\overline{M},\ga_\phi)\to (M,\textsl h)$ is Riemannian and its fibers are $e^{2\phi}$-homothetic to the original fibers in $(\overline M,\ga)$. The main result in this section is Proposition \ref{thm:maingeral}, from which
Theorem \ref{ithm:warped} follows. 

In order to well introduce its statement, define
\begin{gather}
b_k := \dfrac{k+1}{8k},~c_k := \dfrac{(k+1)^2}{8(k-1)k},~\theta_k := 2\dfrac{k-1}{k+1}\label{eq:bk}.
\end{gather}
In what follows, the superscripts consistently denote the dimensions of the underlined manifolds.

\begin{proposition}\label{thm:maingeral}
Let $\pi : F^k \hookrightarrow (\overline M^{n+k},\ga) \to (M^n,\textsl{h})$ be a Riemannian submersion from a closed oriented manifold $\overline M$. Assume that $(\overline M,\ga)$ is locally isometric to a product whose fibers have constant scalar curvature $c \in \mathbb{R}$. 
    \begin{enumerate}[$(a)$]
        \item If $c > 0$ and $\min_{\overline M} \,\mathrm{scal}_{\textsl h} > 0$ then any basic smooth function 
        is the scalar curvature of some general vertical warping metric;
        \item Assume $c \leq 0$ and suppose that a basic function $f : \overline M \to \mathbb R$ satisfies the condition:     \begin{gather}\label{condition1000}
            - b_k\left(f-\mathrm{scal}_{\textsl{h}}\right) + c_kc\left(\mathrm{vol}(M)\right)^{2/\theta_k - 1} \geq 0.
        \end{gather} 
        If $|c| > 0$, then $f$ is the scalar curvature of some general vertical warping metric. In contrast, if $c = 0$, there exists some $A\geq 0$ such that $f+A$ is the scalar curvature of a general vertical warping metric.
    \end{enumerate}
\end{proposition}

Before proceeding to the proof, we observe that the scope of application of Proposition \ref{thm:maingeral} is broader than local products of the form $M\times F$. Indeed, every warped product $(\overline M,\ga)=(M\times F,\textsl h\times_\phi \ga_f)$ admits a metric where it is locally a product. Lemma \ref{thm:maingeral} also holds for submersions which are locally warped products (see Section \ref{sec:examples} for further details).

\subsection{The case $c \leq 0$}

\label{sec:cleq0}
To begin with, assume that $(M^n,\textsl{h})$ is a closed Riemannian manifold and $(F,\ga_F)$ is a Riemannian manifold with constant scalar curvature $c\leq 0$. Let $f: M\times F\to \mathbb{R}$ be a smooth basic function. We start by discussing the proof of Proposition \ref{thm:maingeral} to the case where $\overline M$ is a Riemannian product. To do so, we apply standard variational methods to guarantee the existence of positive solutions for the following PDE:

\begin{equation}\label{eq:EDP}
\Delta_{\textsl{h}}u + \dfrac{k+1}{4k}(f-\mathrm{scal}_{\textsl{h}})u -\dfrac{k+1}{4k} u^{\frac{k-3}{k+1}}c = 0.
\end{equation}

Equation \eqref{eq:EDP} is obtained from a general vertical warping through Lemma \ref{lem:scalarwarped}, a proof of that is obtained as a consequence of the proof of Lemma \ref{lem:scalargeral}, to be stated and proved later in this paper.

\begin{lemma}\label{lem:scalarwarped}
Let $(M,\textsl{h})$ and $(F,\ga_F)$ be Riemannian manifolds and $\phi : M\to \mathbb{R}$ be a smooth function. Denote by $\ga$ the warped metric of $M\times_{e^{2\phi}}F$. Then the scalar curvature of $\ga$ is given by
\begin{equation*}\label{eq:equacao}
\mathrm{scal}_{\ga} = \mathrm{scal}_{\textsl{h}} + e^{-2\phi}\mathrm{scal}_{\ga_F} -k(k-1)|\nabla\phi|_{\textsl{h}}^2 - 2k|\nabla\phi|_{\textsl{h}}^2 - 2k\Delta_{\textsl{h}}\phi.
\end{equation*}
\end{lemma}
Indeed, given a basic function $f\in C^{\infty}(M\times F;\mathbb{R})$, consider the PDE
\begin{equation}\label{eq:eq}
\mathrm{scal}_{\textsl{h}} + e^{-2\phi}\mathrm{scal}_{\ga_F} - f = k(k-1)|\nabla\phi|_{\textsl{h}}^2 + 2k\left\{|\nabla\phi|_{\textsl{h}}^2 + \Delta_{\textsl{h}}\phi\right\}. 
\end{equation}
The solution $\phi$ is such that $f$ is the scalar curvature of the warped metric $\widetilde{\ga} = \textsl{h} + e^{2\phi}\ga_F$ (this is a consequence of the to-be-proved Lemma \ref{lem:scalargeral}). By setting $\phi = \log u^{\frac{2}{k+1}}$ a direct computation shows that equation \eqref{eq:eq} is reduced to equation \eqref{eq:EDP}.

To show that the PDE \eqref{eq:EDP} has a positive solution, assume that $f,\mathrm{scal}_{\textsl{h}}$ are continuous functions. Define the following functional $J : W^{1,2}(M) \to \mathbb{R}$ by
\begin{gather}
J(u)= \dfrac{1}{2}\int_M|\nabla u|_{\textsl{h}}^2 -b_k\int_M\left(f-\mathrm{scal}_{\textsl{h}}\right)u^2 + c_k\int_Mcu^{\theta_k}.\label{eq:functional1}
\end{gather}
Note that $J$ is well defined and, according to the classical Sobolev Embedding Theorem, it is of class $C^1$.

\subsubsection{The case $c < 0$}
Assume that $c < 0$. We will obtain the desired solution $u$ to equation \eqref{eq:EDP} as a minimum for $J$ restricted to the set $\mathbf{M} := \{u \in W^{1,2}(M) : u \geq 0,~ \int_M u^{\theta_k} = 1 \}$. This is the content of Lemma \ref{lem:main}.

\begin{lemma}\label{lem:main}
Define $\mathbf{M} := \{u \in W^{1,2}(M) : u \geq 0,~ \int_M u^{\theta_k} = 1 \}.$ Assume that $f, \mathrm{scal}_{\textsl{h}}$ are continuous functions, that $c< 0$ and that the following inequality holds
\begin{gather}
- b_k\left(f-\mathrm{scal}_{\textsl{h}}\right) + c_kc\left(\mathrm{vol}(M)\right)^{2/\theta_k - 1} \geq 0.
\end{gather} 
Then,
\begin{enumerate}
\item there is a constant $c_0 > 0$ such that $J\big|_{\mathbf{M}} >c_0,$
\item $J\big|_{\mathbf M}$ is coercive,
\item $\mathbf{M}$ is weakly closed and $J\big|_{\mathbf M}$ is weakly lower semi-continuous.
\end{enumerate}
\end{lemma}
Before proving it, let us show how Proposition \ref{thm:maingeral} follows from it. Lemma \ref{lem:main} provides that $J\big|_{\mathbf{M}}$ has a minimum point in $\mathbf{M}$. Therefore, we have a Lagrange Multiplier problem: there exists $\Lambda \in\bb R$ satisfying
\begin{equation}\label{eq:lagrange}
\int_M\langle \nabla u,\nabla v\rangle -2b_k\int_M(f-\mathrm{scal}_{\textsl h})uv + cc_k\theta_k\int_Mu^{\theta_k-1}v = \theta_k\Lambda\left(\int_Mu^{\theta_k-1}v\right),
\end{equation}
$\forall v\in W^{1,2}(M)$.
We claim that $\Lambda\geq 0$: by taking $v =u$, we have
\begin{equation*}
	\int_M|\nabla u|^2-2b_k\int_M(f-\mathrm{scal}_{\textsl h})u^2 + cc_k\theta_k\int_Mu^{\theta_k} = \theta_k\Lambda\int_M u^{\theta_k}
\end{equation*}
However, since $\theta_k \leq 2$ the H\"older's inequality implies 
\begin{equation*}
\int_Mu^{\theta_k}=\left(\int_M u^{\theta_k}\right)^{2/\theta_k} \leq \left(\mathrm{vol}(M)\right)^{2/\theta_k - 1}\int_M u^2.
\end{equation*}

Therefore,
\begin{align*}\label{eq:lagrange2}
	\theta_k\Lambda\int_M u^{\theta_k}&\geq 
	\int_M|\nabla u|^2+2\int_M\left(-b_k(f-\mathrm{scal}_{\textsl h}) + \frac{\theta_k}{2}cc_k(\rm{vol}(M))^{2/\theta_k-1}\right)u^{2} \\\nonumber 
	&\geq 2\int_M\left(-b_k(f-\mathrm{scal}_{\textsl h}) + cc_k(\rm{vol}(M))^{2/\theta_k-1}\right)u^{2}\geq 0,
\end{align*}
so $\Lambda \geq 0$.

We now multiply both sides of \eqref{eq:lagrange} by $\zeta>0$ to get
\begin{equation*}
\int_M\langle \nabla (\zeta u),\nabla v\rangle -2b_k\int_M(f-\mathrm{scal}_{\textsl h})(\zeta u)v + \zeta^{-\theta_k+2}(cc_k-\Lambda)\theta_k\int_M(\zeta u)^{\theta_k-1}v = 0,
\end{equation*}
$\forall v\in W^{1,2}(M)$. But since $\Lambda\geq 0$ we can assume that $\zeta$ is such that $\zeta^{-\theta_k+2}(cc_k-\Lambda)=cc_k$, concluding that $\tilde u=\zeta u$ is a weak solution for equation \eqref{eq:EDP}.

To finish the argument, note that if $f$ and $\mathrm{scal}_{\textsl{h}}$ are smooth functions, the classical regularity theory for elliptic PDEs (see \cite[Theorem 3.58, p. 87]{aubinbook}) guarantees that the solution $u$ is smooth. Finally, we observe that the solution $u$ is positive. Otherwise, a Maximum Principle argument, as in \cite[Proposition 3.75, p.98]{aubinbook} guarantees $u\equiv 0$.
We thus have proved that any basic smooth function $f: M\times F \to \mathbb{R}$ satisfying \eqref{condition1000} can be realized as the scalar curvature of a warped product metric, thus concluding item $(2)$ of Proposition \ref{thm:maingeral} in the case of Riemannian products. It is left to prove Lemma \ref{lem:main}.
\begin{proof}[Proof of Lemma \ref{lem:main}]
It is easy to see $\mathbf{M} \neq \emptyset.$ Moreover, Poincaré's Inequality implies that 
\begin{align}
J(u)&\geq 
	\frac{1}{2}\int_M|\nabla u|^2+\int_M\left(-b_k(f-\mathrm{scal}_{\textsl h}) + cc_k(\rm{vol}(M))^{2/\theta_k-1}\right)u^{2}\label{eq:lagrange3}\\ \nonumber 
	&\geq \int_M\left(\frac{\lambda_1}{2}-b_k(f-\mathrm{scal}_{\textsl h}) + cc_k(\rm{vol}(M))^{2/\theta_k-1}\right)u^{2} - \frac{\lambda_1}{2\rm{vol}(M)}\left(\int_M u\right)^2,
\end{align}
where $\lambda_1$ is the first positive eigenvalue of $-\Delta_{\textsl h}$. On the other hand, the H\"older's Inequality gives:
\begin{gather*}
 \int_M u\leq \rm{vol}(M)\left(\int_M u^{\theta_k}\right)^{1/\theta_k}=\rm{vol}(M),
\end{gather*}
leading to an explicit bound for the last term above. We thus conclude that $J(u)$ has a lower bound, since $-b_k(f-\mathrm{scal}_{\textsl h}) + cc_k(\rm{vol}(M))^{2/\theta_k-1}\geq 0$ by hypothesis. Likewise, coercivity follows directly from \eqref{eq:lagrange3}: If $\|u\|_{L^2}\rightarrow \infty$ then equation \eqref{eq:lagrange3} implies that $J(u) \rightarrow \infty$. If $\|u\|_{L^2}$ is bounded but $\||\nabla u|\|_{L^2} \rightarrow \infty$, the same equation ensures the result.

To see that $\mathbf{M}$ is weakly closed, take $\{u_m\}\subset \mathbf{M}$ weakly converging to $u \in W^{1,2}(M).$ According to the Rellich--Kondrachov's Theorem, the Sobolev space $W^{1,2}(M)$ compactly embeds in $L^{\theta_k}(M)$. Hence, since $1=\int_M|u_m|^{\theta_k}\rightarrow \int_M|u|^{\theta_k}$, we have the result: not only $\int_M |u|^{\theta_k} =1$ but since we have the strong $L^{\theta_k}(M)$ convergence, hence strong $L^1(M)$ convergence, it holds that, up to passing to a subsequence, $\{u_m\}$ pointwise converges to $u$, so $u\geq 0$ almost everywhere. These imply that $\mathbf{M}$ is weakly closed, as desired.

Finally, to ensure that $J$ is weakly lower semi-continuous, note that since $\mathrm{scal}_{\textsl{h}},~\mathrm{scal}_{\ga_F}$ and $f$ are continuous functions and
\begin{equation*}
	\int_M|\nabla u|_{\textsl{h}}^2 \leq \liminf_{m\to\infty}\int_M|\nabla u_{m}|_{\textsl{h}}^2,
\end{equation*}
we conclude that
\begin{align*}
	\liminf_{m\to\infty} J(u_{m}) &= \liminf_{m\to\infty}\left(\int_M|\nabla u_{m}|_{\textsl{h}}^2 - b_k\int_M(f-\mathrm{scal}_{\textsl{h}})u_{m}^2 + c_k\int_Mcu_{m}^{\theta_k}\right)\\
	&\geq \int_M|\nabla u|_{\textsl{h}}^2- b_k\int_M(f-\mathrm{scal}_{\textsl{h}})u^2 + c_k\int_Mcu^{\theta_k} = J(u).\qedhere
\end{align*}
\end{proof}

\subsubsection{The case $c=0$}
We finally deal with the case $c = 0$. We only sketch it since it follows closely to the proof of Lemma \ref{lem:main} as we shall make clear. 

In this case, the functional $J$ is reduced to
\begin{equation}
J(u) = \h\int_M|\nabla u|_{\textsl h}^2 -b_k\int_M(f-\mathrm{scal}_{\textsl h})u^2
\end{equation}

For this case, we can search for a minimum of $J$ restricted to the set
\[\mathbf{M} = \left\{u\in W^{1,2}(M) : \int_Mu^2 = 1\right\}.\]

Since the H\"older's Inequality implies that there is $C>0$ such that for any $u\in \mathbf{M}$ we have $(\int_M u)^{2} \leq C\left(\int u^2\right) = C$, it is immediate to see that the Poincaré's Inequality implies that $J$ has a lower-bound restricted to $\textbf M$. Moreover, the coercivity also follows under the assumption that $f-\mathrm{scal}_{\textsl h}\leq 0$. The same routine argument used in Lemma \ref{lem:main} implies that $\textbf{M}$ is weakly closed. Therefore, $J$ has a critical point in $\textbf{M}$.

Once more, given any function $v\in W^{1,2}(M)$ tangent to $\mathbf{M}$, one has a Lagrange Multiplier problem:
\begin{equation}\label{eq:lagranjeado}
 \int_M\langle \nabla u,\nabla v\rangle -2b_k\int_M(f-\mathrm{scal}_{\textsl h})uv = 2\Lambda \int_M uv,
\end{equation}
for every $v\in W^{1,2}(M)$. Making $v = u$ we see that since $-2b_k\max_M(f-\mathrm{scal}_{\textsl h})\geq 0$ and $\int_M u^2 = 1$ then $\Lambda \geq 0$. 

Now let us rewrite equation \eqref{eq:lagranjeado} as
\begin{equation}\label{eq:reescrita}
 \int_M\langle \nabla u,\nabla v\rangle -2b_k\int_M(f-\mathrm{scal}_{\textsl h}+\Lambda b_k^{-1})uv = 0.
\end{equation}
Note that equation \eqref{eq:reescrita} is the weak formulation of the warped product approach to prescribing the function
$f+\Lambda b_k^{-1}$, which includes the result defining $A := +\Lambda b_k^{-1}$.

We finish this subsection by sketching the proof of the general case of item $(2)$ of Proposition \ref{thm:maingeral} -- it follows via straightforward modifications of the previous arguments. We start by recalling the following result:

\begin{lemma}\label{lem:scalargeral}

Let $F^k\hookrightarrow (\overline M^{n+k},\ga) \stackrel{\pi}{\rightarrow} (M^n,\textsl{h})$ be a Riemannian submersion. Let $\widetilde \ga$ be a general vertical warping metric on $M$ via the function $u^{\frac{4}{k+1}}$, where $u : M\to \mathbb{R}$ is a smooth basic function. Then the scalar curvature of $\widetilde \ga$ is given by
\begin{multline}\mathrm{scal}_{\widetilde \ga} = \mathrm{scal}_{\ga} -\dfrac{4k}{k+1}u^{-1}\Delta_{\textsl{h}}u + \left(4+2(k-1)\right)\frac{2}{k+1}u^{-1}\d u(H)+ \left(u^{-\frac{4}{k+1}}-1\right)\mathrm{scal}_{\ga_F} - \left(1-u^{\frac{4}{k+1}}\right)\|A\|^2,\end{multline}
where $\|A\|^2$ is a term that depends on the non-integrability of the horizontal distribution $\cal H=(\ker \d\pi)^\perp$. In particular, $\|A\|^2=0$ if $\pi$ is locally a warped product.
\end{lemma}

\begin{proof}
 We shall use the formulae described in Chapter 2 in \cite{gw}. Next, $K$ always describes the \emph{non-reduced sectional curvature} for the subscripted metric. Let $T_i = e^{-\phi}v_i,~v_i\in \ker \d \pi$ and $X\in \cal H.$ Then,
\begin{align*}
    {\widetilde K}(e^{-\phi}v_1,e^{-\phi}v_2) &= (e^{-2\phi} -1)K_{\ga_F}(v_1,v_2) + K_{\ga}(v_1,v_2) - |\nabla\phi|^2 \\+\d\phi(\sigma(v_1,v_1) +\sigma(v_2,v_2)),\\
    \widetilde K(X,e^{-\phi}v) &= K_{\ga}(X,v) - (1 -e^{2\phi})|A^*_Xv|^2-
    \mathrm{Hess}~\phi(X,X) \\- \d\phi(X)^2 + 2\d\phi(X)g(S_Xv,v).
\end{align*}
Take a $\ga$-orthonormal basis $\{e_i\}$ to $\mathcal{H}.$ Then,
\begin{equation*}
    \sum_{i,j}\widetilde K(e_i,e_j) = (1-e^{2\phi})\mathrm{scal}_{\textsl h} + e^{2\phi}\mathrm{scal}^{\mathcal H},
\end{equation*}
where $\mathrm{scal}^{\cal H}$ means the corresponding sum of sectional curvatures appearing in a scalar curvature expression, but only restricted to \emph{horizontal vectors}, that is, vectors belonging to $\cal H$.

We also have
\begin{equation*}
    \sum_{r,s}\widetilde K(e^{-\phi}v_r,e^{-\phi}v_s) = (e^{-2\phi} -1)\mathrm{scal}_{\ga_F} + \mathrm{scal}^{\mathcal V} - k(k-1)|\nabla\phi|^2 + 2(k-1)\d\phi(H)
\end{equation*}

\begin{equation*}
    2\sum_{i,r}\widetilde K(e^{-\phi}v_r,e_i) = 2\sum_{i,r}K(e_i,v_r) - 2(1-e^{2\phi})\sum_{i,r}|A^*_{e_i}v_r|^2 - 2k\left(\Delta_{\textsl h}\phi + |\nabla\phi|^2\right) + 4\sum_{r,i}d\phi(e_i)g(S_{e_i}v_r,v_r),
\end{equation*}
where $\mathrm{scal}^{\cal V}$ means the corresponding sum of sectional curvatures appearing in a scalar curvature expression, but only restricted to \emph{vertical vectors}, that is, vectors belonging to $\ker \d \pi$.

Without loss of generality we assume that $\d\phi(e_1) = |\nabla\phi|,$ i.e, $e_1 = \frac{\nabla\phi}{|\nabla\phi|},~\d\phi(e_i) = 0,~i \geq 2$, from where we obtain
\begin{align*}
  4\sum_{r,i}\d\phi(e_i)g(\nabla_{e_i}v_r,v_r) = 4\mathrm{tr}~S_{{\nabla\phi}} =4\d\phi(H).
\end{align*}

So we conclude that
\begin{multline}\label{eq:scalarpreliminar}
    \widetilde{\mathrm{scal}} = \mathrm{scal}_g - 2(1-e^{2\phi})\sum_{i,r}|A^*_{e_i}v_r|^2+(1-e^{2\phi})(\mathrm{scal}_B-\mathrm{scal}^{\cal H})
    +(e^{-2\phi}-1)\mathrm{scal}_{\ga_F} \\- k(k-1)|\nabla\phi|^2- 2k(\Delta_{\textsl h}\phi + |\nabla\phi|^2)+\left(4+2(k-1)\right)\d\phi(H).
\end{multline}

Once equation \eqref{eq:scalarpreliminar} only differs from equation \eqref{eq:eq} by the terms
\[2(1-e^{2\phi})\sum_{i,r}|A^*_{e_i}v_r|^2+ \left(4+2(k-1)\right)d\phi(H),\]
by introducing the change of variables $\phi = \log\varphi$ and $\varphi = u^{\frac{2}{k+1}}$ we conclude that
\begin{multline}\widetilde{\mathrm{scal}} = \mathrm{scal}_{\ga} -\dfrac{4k}{k+1}u^{-1}\Delta_{\textsl h}u + \left(4+2(k-1)\right)\frac{2}{k+1}u^{-1}\d u(H)\\ + \left(u^{-\frac{4}{k+1}}-1\right)\mathrm{scal}_F + \left(1-u^{\frac{4}{k+1}}\right)\left(\mathrm{scal}_{\textsl h}-\mathrm{scal}^{\cal H}-2\sum_{i,r}|A^*_{e_i}v_r|^2\right).\end{multline}
\end{proof}

We proceed by proving that:
\begin{lemma}\label{lem:eliptico}
Let $F^k\hookrightarrow (\overline M^{n+k},\ga) \stackrel{\pi}{\rightarrow} (M^n,\textsl{h})$ be a Riemannian submersion which is locally isometric to a product. If $\Delta_{\ga}$ denotes the Laplace operator on the metric $\ga$, then the restriction of $\Delta_{\ga}$ to basic functions defines a strongly elliptic operator.
\end{lemma}
\begin{proof}
Let $u: \overline M \to \mathbb{R}$ be a basic function. Once $\overline M$ is compact, it is possible to choose a collection of open sets $\{W_l\}\subset \overline M$ trivializing the submersion $\pi$ in the following sense (see \cite{Hermann1960ASC})
\[W_l = U_l\times F,~U_l\subset M.\]
Once $u$ is a basic function,
\[u|_{W_l}(x) = u(x_1,x_2) = u(x_1,x_2'),~\forall x_2,x_2'\in F,~\forall x_1\in U_l.\]
If $\{\psi_l\}$ denotes a partition of unity subordinated to $\{U_l\}$, then
$u = \sum_l\psi_lu$ and there is a well-defined injection
\[\zeta : W^{1,2}(\overline M) \to W^{1,2}(M)\]
\[u \mapsto v,\]
where $v = \sum_lv_l,~v_l(x_1) = \psi_l(x_1)u(x_1,x_2),~\forall x_1\in U_l$.

Since the submersion is locally a product, it follows that its fibers are totally geodesic, so $H = 0$. Hence, by identifying $\zeta u = u$ one has that $\Delta_{\textsl{h}}u = \Delta_{\ga}u$ once $\Delta_{\textsl{h}}u = \Delta_{\ga}u - \d u(H) = \Delta_{\ga}u$ (see \cite[Section 2.1.4, p.53]{gw}). 
\end{proof}

To prove Proposition \ref{thm:maingeral}, item $(2)$, we observe that since the Riemannian submersion, $\pi$ is locally isometric to a product, its fibers are totally geodesic submanifolds, and its horizontal distribution is integrable. In particular, $H \equiv 0$ and $\|A\|= 0$. Hence, according to Lemmas \ref{lem:scalargeral} and \ref{lem:eliptico}, given a basic function $f:\overline M\to \mathbb{R}$, we need to study the following elliptic problem
\begin{equation}\label{eq:pdegeral}
\left(\frac{k+1}{4k}\right)u(f-\mathrm{scal}_{\ga}) = -\Delta_{\textsl{h}}u+\left(\frac{k+1}{4k}\right)\left(u^{\frac{k-3}{k+1}}-u\right)c,
\end{equation}
where $0\geq c = \mathrm{scal}_{\ga_F}.$

The proof of item $(2)$ on Proposition \ref{thm:maingeral} follows by adapting the proof of Lemma \ref{lem:main} to solve the PDE \eqref{eq:pdegeral}. This is done by searching for a minimum point in the set \begin{equation*}
\mathbf{M}_{\cal F} := \{u \in W^{1,2}(\overline M) : u\geq 0, ~u~\text{is basic and}~ \int_M u^{\theta_k} = 1\}.
\end{equation*}
{We use the subscript $\cal F$ to emphasize that we are considering classes of elements whose function representatives are constant along the foliation induced by the fibers $F$.}

\subsection{The case $c > 0$} We consider this case separately since we need a different version of Lemma \ref{lem:main}. Our result is a natural generalization of \cite[Theorem 4.10, p. 253]{ehrlich1996}. We first remark that this case suffers no obstruction as it is substantiated by the fact that a product of a manifold with positive scalar curvature always admits positive scalar curvature.

Following the notation established in Section \ref{sec:cleq0} and motivated by Lemma \ref{lem:eliptico}, we consider the following functional of class $C^1$:
\begin{equation}
J(u) = \frac{1}{2}\int_M|\nabla u|^2 +2b_k\int_M\left\{\left(\mathrm{scal}_{\ga}-c-f\right)\frac{u^2}{2}+c\theta_k^{-1}u^{\theta_k}\right\}
\end{equation}
defined in
\begin{equation*}
\mathbf{M}_{\cal F} := \{u \in W^{1,2}(\overline M) : u\geq 0, ~u~\text{is basic and}~ 2b_kc\theta_k^{-1}\int_M u^{\theta_k} = 1\}.
\end{equation*}
By following arguments analogous to the ones in the proof of Lemma \ref{lem:main}, it can be shown that if 
\[\min_M(\mathrm{scal}_{\ga}-c) \geq \max_Mf,\]
then $J|_{\mathbf{M}_{\cal F}}$ has a minimum point that is a smooth, positive solution to equation \eqref{eq:pdegeral}. The item is proved since we can prescribe any smooth function: Dividing $f$ by an arbitrary positive constant and scaling the resulting metric is always possible, so the above constraint is always overcome.\halmos

\section{Examples}
\label{sec:examples}

\subsection{The constructions on Theorem \ref{thm:exemplos1}}
The $\star$-diagram construction was used in \cite[Theorem 1.1]{SperancaCavenaghiPublished} to produce several examples of manifolds satisfying the hypotheses of $F$ in Theorem \ref{thm:scalar}. We recall them here:

\begin{theorem}\label{thm:llohann}
Let $\Sigma^7$ and $\Sigma^8$ be any homotopy spheres of dimensions 7 and 8, respectively; $\Sigma^{10}$ be any homotopy 10-sphere which bounds a spin manifold; $\Sigma^{4m{+}1},\Sigma^{8m+5}$ be Kervaire spheres of dimensions $4m{+}1,8m{+}5$, respectively. Then, the following manifolds admit an explicit realization as in diagram \eqref{eq:CD} and satisfy the hypotheses of $F$ in Theorem \ref{thm:scalar}:
\begin{enumerate}[$(i)$]
\item $M^7\#\Sigma^7$, where $M^7$ is any 3-sphere bundle over $S^4$
\item $M^8\#\Sigma^8$, where $M^8$ is either a 3-sphere bundle over $S^5$ or a 4-sphere bundle over $S^4$
\item $M^{10}\#\Sigma^{10}$, where:
\begin{enumerate}[$(a)$]
\item $M^{10}=M^8\times S^2$ with $M^8$ as in item $(ii)$
\item $M^{10}$ is any 3-sphere bundle over $S^7$, 5-sphere bundle over $S^5$ or 6-sphere bundle over $S^4$
\end{enumerate}
\item $M^{4m+1}\#\Sigma^{4m+1}$ where \label{item:5}
\begin{enumerate}[$(a)$]
\item $S^{2m}\hookrightarrow M^{4m+1}\to S^{2m+1}$ is the sphere bundle associated to any multiple\footnote{\label{footnote1}That is, a bundle whose transition function $\alpha: S^{n-1}\to G$ is a multiple of $\tau_{2m}: S^{2m}\to O(2m+1)$, $\tau_m^\bb C: S^{2m}\to U(m)$ or $\tau_m^\bb H: S^{4m+2}\to Sp(m)$, for $G=O(2m), U(m+1)$ or $Sp(m)$, the transition functions of the orthonormal frame bundle and its reductions, respectively.} of $O(2m{+}1)\hookrightarrow O(2m{+}2)\to S^{2m+1}$, the frame bundle of $S^{2m+1}$
\item $\bb C \rm P^{m}\hookrightarrow M^{4m+1}\to S^{2m+1}$ is the $\bb C\rm P^m$-bundle associated to any multiple of the bundle of unitary frames $U(m)\hookrightarrow U(m+1)\to S^{2m+1}$
\item $M^{4m+1}=\frac{U(m+2)}{SU(2)\times U(m)}$
\end{enumerate}
\item $(M^{8r+k}\times N^{5-k})\#\Sigma^{8r+5}$ where $N^{5-k}$ is any manifold with positive Ricci curvature and
\begin{enumerate}[$(a)$]
\item $S^{4r+k-1}\hookrightarrow M^{8r+k}\to S^{4r+1}$ is the $k$-th suspension of the unitary tangent $S^{4r-1}\hookrightarrow T_1S^{4r{+}1}\to S^{4r+1}$,
\item for $k=1$, $\bb H\rm P^{m}\hookrightarrow M^{8m+1}\to S^{4m+1}$ is the $\bb H\rm P^m$-bundle associated to any multiple of $Sp(m)\hookrightarrow Sp(m+1)\to S^{4m+1}$
\item for $k=0$, $M=\frac{Sp(m+2)}{Sp(2)\times Sp(m)}$
\item for $k=1$, $M=M^{8m+1}$ is as in item $(\ref{item:5})$
\end{enumerate}
\end{enumerate}
\end{theorem}

In what follows, we denote by $F\hookrightarrow P\times_G F\to B$ the bundle associated to $G\hookrightarrow P\to B$ with fiber $F$. 

Consider $P, M, M', G$ as in diagram \eqref{eq:CD}. If $M$ can be a fiber for Theorem \ref{thm:scalar}, so can $M'$. Therefore, Theorem \ref{thm:scalar} can be applied to the following bundles:
\begin{enumerate}
\item The associated bundle $M \hookrightarrow P\times_GM \to M$ to $\pi : P\to M,$
\item The associated bundle $M' \hookrightarrow P\times_GM' \to M$ to $\pi : P \to M,$
\item The associated bundle $M \hookrightarrow P\times_GM \to M'$ to $\pi' : P\to M',$
\item The associated bundle $M' \hookrightarrow P\times_GM' \to M'$ to $\pi' : P\to M'$.
\end{enumerate}

Here we consider $M, M'$ with the $G$-actions descending from $\star,\bullet$, respectively. Note that $M'$ can be taken as any of the manifolds given by Theorem \ref{thm:llohann} and $M$ as its ``counterpart'' in its corresponding $\star$-diagram provided in \cite{SperancaCavenaghiPublished}. Namely, $M$ is the original manifold, and $M'$ is the resulting connected sum.

 Now we explore more explicitly the construction in the previous paragraph. Consider Duran's $\star$-diagram diagram introduced in \cite{duran2001pointed}, i.e., the $\star$-diagram with $G = S^3, P = Sp(2)$, $M = S^7, M' = \Sigma^7$. Let $B=S^7$ or $\Sigma$. Consider its respective principal bundle 
$\mathrm{pr} : S^3\hookrightarrow Sp(2) \to B.$

 Let $Sp(2) \hookrightarrow Sp(2)\times_{S^3}Sp(2) \to B$ be the associated bundle to $\mathrm{pr} : S^3 \hookrightarrow Sp(2) \to B$. In both cases, the $S^3$-manifold $B$ can be regarded as the fiber of associated bundles to the principal bundles $\pi : Sp(2) \to S^7$ or $\pi': Sp(2) \to \Sigma^7$, producing new examples of fiber bundles where Theorem \ref{thm:scalar} can be applied to:
\begin{enumerate}[(a)]
\item $S^7 \hookrightarrow Sp(2)\times_{S^3}S^7 \to \Sigma^7$;
\item $\Sigma^7 \hookrightarrow Sp(2)\times_{S^3}\Sigma^7 \to \Sigma^7$;
\item $\Sigma^7\hookrightarrow Sp(2)\times_{S^3}\Sigma^7\to S^7$;
\item $S^7 \hookrightarrow Sp(2)\times_{S^3}S^7 \to S^7$.
\end{enumerate}

The procedure is done similarly for all the other examples in the thesis of Theorems \ref{thm:exemplos1} and \ref{thm:exemplos2}. We recommend the Examples section in \cite{cavenaghi2019positive} for additional results/clarifications.

\subsection{Polar foliations and Calabi--Yau bundles}

\label{sub:exemplinhos}
Observe that a manifold is locally a metric product if it admits some cover that splits as a metric product. Specifically,  $(\overline{M},\ga)$ is locally isometric to a product if there are  $(M,\textsl{h})$ and $(F,\ga_F)$ such that $\overline M$ is isometric to the quotient of $(M{\times}F, \textsl h\times \ga_T)$ by a subgroup $\Gamma<\pi(\overline M)$ acting as deck transformations.
Although being locally isometric to a product is quite restrictive, we can slightly weaken it by considering locally warped manifolds. A manifold is locally a warped product if there is a function $\psi$ such that the general vertical warping induced by $\psi$, $\widetilde \ga$, makes $(\overline M, \widetilde\ga)$ locally a metric product. 
We recall a result in Gromoll--Walschap \cite{gw} to characterize the latter.
\begin{proposition}[Proposition 2.2.1, \cite{gw}]\label{prop:gwlocallywarped}
Let $\pi:\overline M^{n+k}\to M^n$ be a Riemannian submersion. Then $\pi$ is locally a warped product if and only if
\begin{enumerate}
\item the distribution $\cal H=(\ker \d\pi)^\perp$ is integrable;
\item the fibers are totally umbilic submanifolds of $\overline M$;
\item $\pi$ is isoparametric; i.e., the mean curvature form $\kappa$ (dual to the mean curvature vector field) is basic.
\end{enumerate}
\end{proposition}
Since the metric on the factor $(M,\textsl h)$ is preserved by this construction, Proposition \ref{thm:maingeral} can be applied to manifolds satisfying the conditions in Proposition \ref{prop:gwlocallywarped}.

\subsubsection{Calabi--Yau bundles}
We proceed with a brief discussion about prescribing scalar curvature on some Calabi--Yau bundles.
We first observe that realizing basic scalar functions on a trivial bundle $Y\times F\to Y$ follows from previous works: Assume that $F$ has a metric $\ga_F$ with constant scalar curvature $c$. If $f:Y\to \bb R$ is a smooth function, consider a metric $\ga$ on $Y$ that realizes $\scal_{\ga}=f-c$. Then, the scalar curvature of the product $\tilde\ga=\ga\times\ga_F$ satisfy $\scal_{\tilde\ga}(x,y)=f(x)$. The originality of Theorem \ref{ithm:warped} and Proposition \ref{thm:maingeral} is the possibility to fix and preserve an initial metric $\ga$ on the base $Y$. 

An exciting application revolves around Calabi--Yau manifolds.
Specifically, we apply Proposition \ref{thm:maingeral} to the following result of Tosatti--Zhang:
\begin{theorem}[Theorems 1.2 and 1.3 in \cite{tosatti}]\label{thm:tosatti}
Suppose $\pi:X\to Y$ is a holomorphic submersion with connected fiber $F$ satisfying one of the following:
\begin{enumerate}
\item $X,Y$ are projective manifolds with $X$ Calabi--Yau;
\item $X,Y$ are compact K\"ahler manifolds; $Y,F$ are Calabi--Yau; and either $b_1(F)=0$ or $b_1(Y)=0$ and $F$ is a torus.
\end{enumerate}
Then $Y$ is Calabi--Yau, and there is a finite unramified cover $p : \widetilde Y \to Y$ such that the pullback bundle $\tilde \pi: p^*X\to \widetilde Y$ is holomorphically trivial.
\end{theorem}

In the present context, a \textit{finite unramified covering map} means a smooth covering map with a finite number of sheets.
Theorem \ref{thm:tosatti} is proved in the holomorphic context without considering compatibility between the metrics and the submersion. We observe here that the holomorphic submersion in the theorem can be taken as a Riemannian one while choosing a Calabi--Yau metric $\ga_Y$ on $Y$. 

To this aim, note that since $Y$ is Calabi--Yau, its fundamental group is Abelian. Therefore every covering is regular. In particular, $p:\widetilde Y\to Y$ is a principal bundle with a finite principal group $\Gamma=\pi_1(Y)/\pi_1(\widetilde Y)$. Since $\Gamma$ is finite, there is a $\Gamma$-invariant metric $\ga_F$ on $F$, therefore there is a metric $\ga$ on $X$ such that both $\tilde\pi:(p^*X,p^*\ga_Y\times \ga_F)\to (X,\ga)$ and $\pi:(X,\ga)\to (Y,\ga_Y)$ are Riemannian submersions. We thus conclude that $\ga$ is locally a metric product, and Theorem \ref{ithm:warped} applies. The proof of Corollary \ref{icor:CY} is concluded by observing that the construction carries over with any arbitrary metric on $Y$.

\subsection*{Acknowledgments}
The authors are thankful to Prof. Marcus Marrocos for his valuable comments on the analytical part of this paper and to Prof. Marcos Alexandrino for pointing out the possible application of our results to polar foliations.

\appendix

\section{General Vertical Warpings}

\label{sec:appendix}

Let $F\hookrightarrow (M,g) \stackrel{\pi}{\rightarrow} B$  be a Riemannian submersion. Obtaining new Riemannian submersions from $\pi$ is possible by introducing some metric deformations changing the $g$ in vertical directions. Let $\phi: M \to \mathbb{R}$ be a smooth function. We define a new metric $\widetilde g$ on $\pi$ in the following way 
\[\widetilde g(E,F) :=g\left(E^{\cal H},F^{\cal H}\right) +e^{2\phi}g\left(E^{\cal V},F^{\cal V}\right),~\forall  E,F\in T_pM,~\forall p\in M.\]
Since this metric preserves the horizontal distribution,  $\pi : (M, \widetilde g) \to B$ remains a Riemannian submersion. Denote by $\widetilde \nabla, \widetilde R$ the Levi-Civita connection and the Riemann curvature tensor of $\widetilde g.$ 

\begin{definition}
A Riemannian submersion $\pi : (M,\widetilde g)\to B$ such that $\phi$ is constant along the fibers (or equivalently, $\nabla\phi$ is basic) is called \textit{general vertical warping}.
\end{definition}

\begin{proposition}
\label{prop:generalverticalwarping}
Let $F\hookrightarrow (M,g) \stackrel{\pi}{\rightarrow} B$  be a Riemannian submersion and $\widetilde g$ be a general vertical warping of $g$ with respect to the function $e^{2\phi},~\phi \in C^{\infty}(M;\mathbb{R}).$ Fix $p\in M$. Let $X, Y \in \mathcal H_p,~V, V_i \in \mathcal V_p,~i\in \{1,2\}.$  If $g(V_1,V_2) = 0,$ the following formulae hold true for the sectional curvature $\widetilde K$ of $\widetilde g$:
\begin{center}
\begin{equation*}
    \widetilde K(X,Y) = (1-e^{2\phi})K_B(X,Y) + e^{2\phi}K(X,Y),
\end{equation*}
\begin{multline*}
    \widetilde K(V_1,V_2) = (e^{2\phi}-e^{4\phi})K_F(V_1,V_2) + e^{4\phi}K_g(V_1,V_2)  
    -e^{4\phi}|V_1|^2|V_2|^2|\nabla\phi|^2 + e^{4\phi}d\phi(\sigma(V_1,V_1))|V_2|^2 +e^{4\phi}d\phi(\sigma(V_2,V_2))|V_1|^2,
    \end{multline*}
    \begin{multline*}
    \widetilde K(X,V) = K_g(X,V)e^{2\phi} - e^{2\phi}\left(1-e^{2\phi}\right)|A^*_XV|^2 -
    \left\{\mathrm{Hess}~\phi(X,X) + d\phi(X)^2\right\}e^{2\phi}|V|^2 +2e^{2\phi}d\phi(X)g(S_XV,V).
    \end{multline*}
    \end{center}
    \end{proposition}
\begin{proof}
See \cite[Section 2.1.3, p. 52]{gw}
\end{proof}
\subsection{Warped products}

Let $(B,g_B)$ e $(F,g_F)$ be Riemannian manifolds. Assume that $\pi: M = B\times F\to B$ is a trivial Riemannian submersion with the product metric on $M$. Let $\phi: B\to \mathbb{R}$ be a smooth function. Then the metric $\widetilde g:= g_B\times e^{2\phi}g_F$ is an example of general vertical warping known as \textit{warped product}. The Riemannian manifold $(M,\widetilde g)$ is called warped product of $B$ and $F$, being usually denoted by $B\times_{e^{2\phi}}F$.  

On warped products, the Gray--O'Neill tensor $A$ vanishes identically. Moreover, the second fundamental form of the fibers satisfy
\begin{equation}\label{eq:shapenowarped}
    \widetilde\sigma\left(T_1,T_2\right) = -e^{2\phi}g\left(T_1,T_2\right)\nabla\phi = -\widetilde{g}\left(T_1,T_2\right).
\end{equation}

The following formulae for the sectional curvature of a warped product hold true
\begin{proposition}\label{prop:formulaewarped}
Consider a warped product $\pi : \left(B\times F,\widetilde g = g_B\times e^{2\phi}g_F\right) \to M$ and let $\widetilde K$ be the sectional curvature of $\widetilde g$. Fix $(p,f) \in M\times F$ and take $X, Y\in \mathcal H_p,~ V, V_i \in \mathcal V_p,~i\in \{1,2\}.$ Then
\begin{align}
    \widetilde K(X,Y) &= K_{B}(X,Y);\\
    \widetilde K(V_1,V_2) &= e^{2\phi}\left\{K_F(V_1,V_2) - e^{2\phi}|\nabla \phi|^2\left(|V_1|^2|V_2|^2 - \langle V_1,V_2\rangle\right)^2\right\};\\
    \label{eq:warpedvertizon}\widetilde K(X,V) &= -e^{2\phi}|V|^2\left(\langle \nabla\phi,X\rangle^2 + \mathrm{Hess}~\phi(X,X)\right),\\
\end{align}
\end{proposition}

\subsection{Canonical deformation}
Another simple case of general vertical warping happens when one takes $\phi(p) = t\in \mathbb{R},~\forall p\in M$. The metric $\widetilde g$ is usually known as the canonical variation of $g$. Let $\widetilde g = g\Big|_{\cal B} + e^{2t}g\Big|_{\cal V}.$

	\begin{proposition}\label{prop:curvatures}
		Let $\pi: F \hookrightarrow (M,g) \to B$ be a Riemannian submersion with totally geodesic fibers. Let $\widetilde K,~K,~K_B,~K_F$ denote the non-reduced sectional curvatures of $\widetilde g,~g_B,~g_F$, respectively, where $\widetilde g$ is the canonical variation  of $g;$ $g_B$ is the submersion metric on $B$, and $g_F$ the metric on $F$. Then, if $X,Y,Z \in \mathcal{H},$ and  $V,W \in \mathcal{V},$
		\begin{enumerate}
			\item $\widetilde K(X,Y) = K_B(\pi_{\ast}X,\pi_{\ast}Y)(1-e^{2t}) + e^{2t}K(X,Y),$\\
			\item $\widetilde K(X,V) = e^{4t}|A^*_XV|^2,$\\
			\item $\widetilde K(V,W) = e^{2t}K(V,W),$\\
			\item $\widetilde R(X,Y,Z,W) = e^{2t}g((\nabla_XA)_YZ,W).$
		\end{enumerate}
	\end{proposition}

\bibliographystyle{alpha}
	\bibliography{main}

\end{document}